\documentclass[11pt]{amsart}
\usepackage{amssymb,latexsym,graphicx, amscd, enumitem}

\newtheorem{theorem}{Theorem}[section]
\newtheorem{prop}[theorem]{Proposition}

\newtheorem{question}[theorem]{Question}
\newtheorem{definition}[theorem]{Definition}
\newtheorem{cor}[theorem]{Corollary}

\DeclareMathOperator{\C}{\mathbb C}
\DeclareMathOperator{\R}{\mathbb R}
\begin{document}

\title[From Smooth to Almost Complex]{From Smooth to Almost Complex}

\author{Weiyi Zhang}
\address{Mathematics Institute\\  University of Warwick\\ Coventry, CV4 7AL, England}
\email{weiyi.zhang@warwick.ac.uk}

\begin{abstract} 
This article mainly aims to overview the recent efforts on developing algebraic geometry for an arbitrary compact almost complex manifold. 

We review the results obtained by the guiding philosophy that a statement for smooth maps between smooth manifolds in terms of  Ren\'e Thom's transversality should also have its counterpart in pseudoholomorphic setting without requiring genericity of the almost complex structures.  
These include intersection of compact almost complex submanifolds, structure of pseudoholomorphic maps, zero locus of certain harmonic forms, and eigenvalues of Laplacian. In addition to reviewing the compact manifolds situation, we also extend these results to orbifolds and non-compact manifolds. Motivations, methodologies, applications, and further directions are discussed.

The structural results on the pseudoholomorphic maps lead to a notion of birational morphism between almost complex manifolds. This motivates the study of various birational invariants, including Kodaira dimensions and plurigenera, in this setting.  Some other aspects of almost complex algebraic geometry in dimension $4$ are also reviewed. These include cones of (co)homology classes and subvarieties in spherical classes.

\end{abstract}
 \maketitle

\tableofcontents
\section{Introduction}
An almost complex structure $J$ on a smooth manifold $M$ is a linear complex structure on each tangent space, which varies smoothly on $M$. In other words, $J: TM\rightarrow TM$ is an endomorphism of the tangent bundle such that $J^2=-1$. The pair $(M, J)$ is called an almost complex manifold. The geometry of almost complex manifolds has been an important subject of differential geometry since early last century. It contains many outstanding problems, from the famous question on the (non-)existence of complex structures on $S^6$ and Yau's question whether an almost complex manifold admits an integrable complex structure when complex dimension is greater than two \cite{Yau}, Kodaira-Spencer's question on almost Hermitian Hodge numbers that are independent of the metric \cite{Hir} \footnote {This question is recently solved by Tom Holt and the author in \cite{HZ}.}, and Goldberg's conjecture that every complex almost K\"ahler-Einstein manifold is K\"ahler \cite{Gold}, to more recently Donaldson's tamed to compatible question in dimension  $4$ \cite{Don06}. Geometry of almost complex manifolds constitutes one out of three sections in Hirzebruch's 1954 problem list on differentiable and complex manifolds \cite{Hir}. 

A pseudoholomorphic curve, first studied in \cite{NW},  is a smooth map from a Riemann surface into an almost complex manifold $(M, J)$ that satisfies the Cauchy-Riemann equation. Gromov \cite{Gr} first introduced this notion as a fundamental tool to study global properties of symplectic manifolds. It has since revolutionized the field of symplectic topology and greatly influenced many other areas such as algebraic geometry, string theory, and $4$-manifolds. In his Abel Prize interview \cite{RS}, Gromov says that the introduction of pseudoholomorphic curves is unquestionably the particular result he is most proud of.  

The theory of pseudoholomorphic curves behaves very much like the curves theory for algebraic varieties, in particular when $J$ is tamed by a symplectic form or almost K\"ahler, {\it e.g.} in the Gromov-Witten theory. One key reason for that is almost complex structures are rigid objects and pseudoholomorphic curves are governed by an elliptic equation. Hence, it is reasonable to expect theories of almost complex manifolds being parallel to complex algebraic geometry, at least when we have tameness or almost K\"ahler assumption. However, this connection is not so straightforward, as in fact people used to view objects in these two sides from different perspectives. 

Majority of research in the theory of pseudoholomorphic curves,  like the Gromov-Witten theory, takes the ``mapping into" viewpoint.  On the other hand, in the complex setting, most classical research takes the  ``mapping out" viewpoint: namely, a complex submanifold or subvariety is by definition locally the zero locus of analytic functions in terms of complex coordinates. However, this does not work for a general almost complex manifold $(M, J)$, as the image of a $J$-holomorphic curve (we call it a $J$-holomorphic $1$-subvariety) is in general not the zero locus of pseudoholomorphic sections of any complex line bundle over $M$, even locally. 

Hence, we define $J$-holomorphic subvarieties using the ``mapping into" viewpoint.  A $J$-holomorphic subvariety in an almost complex manifold $(M, J)$ is a finite set of pairs $\{(V_i, m_i), 1\le i\le m\}$, where each $V_i$ is an irreducible $J$-holomorphic subvariety and each $m_i$ is a positive integer. Here an irreducible $J$-holomorphic subvariety is the image of a somewhere immersed pseudoholomorphic map $\phi: X\rightarrow M$ from a compact connected smooth almost complex manifold $X$. In particular, a zero-dimensional subvariety is a collection of finitely many points with positive weights (or multiplicities). We call a subvariety an $n$-subvariety if all $V_i$ have Hausdorff dimension $n$. We write the push forward of the fundamental classes of $V_i$ under $\phi_i$ as $e_{V_i}$ (or $[V_i]$ if there is no confusing). 

However, to develop algebraic geometry for almost complex manifolds, we also need to study the ``mapping out" approach. As we have seen in the complex setting, it is essentially the intersection theory of almost complex submanifolds. 

In the smooth category, we  have intersection theory, but it behaves best only when we assume Thom's transversality. Smooth and complex intersection theories are different from each other in two aspects. First, in smooth category, things work under certain genericity assumptions, one most important version of these is Thom's transversality. Meanwhile, in complex analytic setting, the theory works for an arbitrary complex structure. Second, intersection theory in complex category has ``positivity". Namely, the intersection of any two complex subvarieties is a complex subvariety (possibly containing many components with different dimensions) with the right orientation. However, this is not true for intersection of two smooth submanifolds even if the intersection is a smooth manifold. The simplest example is that the zeros of a holomorphic function have positive multiplicities, while the multiplicity of an isolated zero of a smooth function could be any integer. 

Almost complex category shares both features with complex category when we look at intersections of pseudoholomorphic curves in a $4$-dimensional almost complex manifold. It is a fundamental result in the theory of pseudoholomorphic curves that the intersection of distinct irreducible $J$-holomorphic curves is constituted of isolated intersection points with positive local intersection index  \cite{Gr, McD, MW}. 

The above circle of thinking motivates the following philosophy \cite{ZCJM}

\medskip

{\it a statement for smooth maps between smooth manifolds in terms of  R. Thom's transversality should also have its counterpart in pseudoholomorphic setting without requiring the transversality or genericity, but using the notion of pseudoholomorphic subvarieties, in particular when such a statement is available in complex analytic setting.}

\medskip

\noindent 

This paper will review a few, among many more, such directions guided by this philosophy, as well as some applications of these results. These directions include intersection of compact almost complex submanifolds, structure of pseudoholomorphic maps, zero locus of certain harmonic forms, and eigenvalues of Laplacian. To give the readers a comprehensive introduction, we explain the motivations, the connections with other fields, and further directions.

Some other aspects of almost complex algebraic geometry are also reviewed. The structural results on the pseudoholomorphic maps suggest degree one pseudoholomorphic maps are birational morphisms between almost complex manifolds. As a step to develop birational geometry for almost complex manifolds, we define and study of various birational invariants, including Kodaira dimensions and plurigenera, in this setting.  Some other aspects of almost complex algebraic geometry for almost complex $4$-manifolds are also reviewed. These include Mori's cone theorem, the Nakai-Moishezon and the Kleiman type dualities, and subvarieties in sphere classes.  We also give some very preliminary discussions on the  connections with Ruan's symplectic birational geometry \cite{Ru, LR}. 

Another goal of this paper is to provide a toolbox and a taster how these tools are working together to study algebraic geometry of almost complex manifolds. Hence, in addition to introducing results, we also explain the main ideas in the proofs and how these can be used to tackle other problems.  Although our original papers, {\it e.g.} \cite{ZCJM, CZ, CZ2, BonZ}, mainly focus on the compact manifolds without boundary (except in the intersection theory the ambient manifold is allowed to be non compact), our techniques apply also to orbifolds and non-compact manifolds. Some of these generalizations are discussed, although in many cases we do not attempt to phrase the most optimal statements. 

The author would like to thank the referee for suggestions to improve the exposition of the paper.

\section{Intersection theory of almost complex submanifolds}
We start this section with a toy model already mentioned in the introduction, {\it i.e.} the multiplicity of zeros for a continuous function $u:D^2\rightarrow \mathbb R^2$ from the open unit disk $D^2$ that generalizes the multiplicity of zeros of a holomorphic function. First, to make sense of it, we need to assume $u$ has isolated zeros, or more generally, $u$ is admissible, which means $\overline{u^{-1}(0)}\cap \partial D^2=\emptyset$. For example, the function $u(x, y)=x$ is not admissible, while all non-trivial (anti-)holomorphic functions and the function $u(z)=|z|^2$ are admissible. For an admissible function $u$, we can define the ``sum" of multiplicities of its zeros in $D^2$, $I(u)$. 

The multiplicity $I(u)$ should satisfy the following $5$ properties.

\begin{itemize}
\item {\bf Zero}: $I(u)=0$ if $u(a)\ne 0, \forall a\in D^2$.

\smallskip
\smallskip

\item {\bf Invariance under homotopy}: If $u_0, u_1: D^2\rightarrow \R^2$ are admissible and homotopic via an admissible family $u_t$, then $I(u_0)=I(u_1)$.

\smallskip
\smallskip

\item {\bf Degree $k$ map}: If $\theta: D^2\rightarrow D^2$ is a proper degree $k$ map, then $I(u\circ \theta)=k\cdot I(u)$. 
\smallskip
\smallskip

\item {\bf Additivity}: If all the zeros are included in disjoint union $\cup_iD_i \subset D^2$ where each $D_i=\theta_i(D^2)$ with embedding $\theta_i: D^2\rightarrow D^2$, then $I(u)=\sum I(u\circ \theta_i)$.

\smallskip
\smallskip

\item {\bf Generalization}: If $u$ is holomorphic, $I(u)$ is the usual multiplicity of zeros for holomorphic functions. 
\end{itemize}
  
We can define a version of such $I(u)$ using Jacobian $ J_u=\begin{pmatrix} \frac{\partial u^1}{\partial x}&  \frac{\partial u^1}{\partial y} \\
\frac{\partial u^2}{\partial x} & \frac{\partial  u^2}{\partial y}
\end{pmatrix}.$
 We perturb $u$ to an admissible function $\tilde u: D^2 \rightarrow \mathbb R^2$ such that the Jacobian of each zero of $\tilde u$ is non-degenerate. 
 Then $I(u)$ is defined as the sum of the signs of $\det J_{\tilde u}$ at each zero of $\tilde u$. The multiplicity $I(u)$ is independent of the choice of the perturbation $\tilde u$. In fact, this is the only possible $I(u)$ which satisfies all the five natural properties above (see \cite{BonZ}).  
 
It is easy to check that for the function $u(z)=|z|^2$, the multiplicity $I(u)=0$. 

From the above discussion, the zero locus of a function $u$, which could be viewed as the intersection of the graph of $u$ in the trivial rank $2$ bundle with its zero section, is a perfect toy model for our philosophy. First, for a continuous function, only a generic function has isolated zeros. In fact, by Whitney extension theorem, any closed subset is the zero set of a smooth function. For these functions with isolated zeros, the multiplicity of each zero is well defined. While for a non-trivial holomorphic function, the zeros are isolated and the multiplicity is positive, thus a $0$-subvariety. In the almost complex category, we can endow any almost complex structure on the total space $D^2\times \R^2$ where $D^2\times \{0\}$ is an almost complex submanifold.  Then an ``almost complex" function is just a pseudoholomorphic curve intersecting each fiber uniquely. By positivity of intersection for pseudoholomorphic curves, the zero locus is a $0$-subvariety.

The following proof of Liouville's theorem that any holomorphic function on a compact Riemann surface is  constant is a nice example how the positivity of intersection would lead to uniqueness results by applying the five properties of multiplity $I(u)$.

Denote the compact Riemann surface by $S$. The graph $\Gamma_f$ of a holomorphic function $f$ is a complex submanifold in $S\times \C$. 
This is homotopic to the graph of any constant function $\Gamma_c$.  On one hand, the intersection number $\iota(\Gamma_f, \Gamma_0)$ of $\Gamma_f$ and the zero section $\Gamma_0$ is nonnegative and it is positive when $f$ has a zero. On the other hand, $\iota(\Gamma_f, \Gamma_0)=\iota(\Gamma_c, \Gamma_0)=0$ for any $c\ne 0$. 
This implies $\Gamma_f\cap \Gamma_0=\emptyset$ unless $f$ is identically $0$. Similarly, $\Gamma_f\cap \Gamma_c=\emptyset$ unless $f$ is a constant $c$ for $\forall c\in \C$. This would imply $f$ has to be a constant function.

In the almost complex setting, Gromov \cite{Gr} has observed a positivity of intersection for curves and divisors. In our language, it is 

\begin{theorem}[Gromov: dim 2 $\cap$ codim 2]\label{Gromov}
Suppose that $Q$ is a compact codimension two $J$-holomorphic submanifold of the almost complex manifold $(M, J)$, and let $u: D^2\rightarrow (M, J)$ be a $J$-holomorphic curve such that $u(0)\in Q$ and $u(D^2)\subset Q$. Then $u^{-1}(Q)$ is a $0$-dimensional subvariety. 
\end{theorem}
In particular, when $u$ is an inclusion, it implies the intersection $D\cap Q$ is a $0$-subvariety. 

When the ambient manifold is of dimension $4$, we have Micallef-White's more general positivity of intersection \cite{MW} which says, in our terminology, that the intersection of two non-identical irreducible pseudoholomorphic subvarieties is a pseudoholomorphic $0$-subvariety.

\begin{theorem}[Micallef-White]\label{MicW}
Let $u_i: D_i\rightarrow M^4$ be  $J$-holomorphic disks with an intersection point $u_1(p_1)=u_2(p_2)$. Suppose  $u_1(D_1)\ne u_2(D_2)$. Then this is an isolated intersection point with positive local intersection index.
\end{theorem}

To summarize, all the positivity of intersection results obtained so far in pseudoholomorphic setting are about the intersection of a curve and a divisor, {\it i.e.} a pseudoholomorphc subvariety of real codimension two, which are $0$-dimensional subvarieties. 

We now move to submanifolds or subvarieties of other dimensions. In smooth and complex categories, the intersection results work for any dimension. In smooth category, Thom's transversality theorem says if two submanifolds intersect transversely, then the intersection is a smooth manifold. Its complex counterpart is the intersection theory of analytic subvarieties.

The results of the intersection of almost complex submanifolds in \cite{ZCJM} are under the assumption  that there is no ``excess intersection" phenomenon besides the trivial inclusion case. For this reason, we require one of the submanifolds to be of codimension $2$. I believe this is just a technical condition, since excess intersection already appears in complex setting. For example,  projective subspaces of complex dimensions $k$ and $l$ in $\mathbb CP^n$ could share common projective subspaces of any dimension no less than $k+l-n$. Moreover, the intersection could have irreducible components with unequal dimensions.

Taking into account of the above, it is natural to first study the intersection of a pseudoholomotphic subvariety and a divisor, {\it i.e.} no excess intersection case. 
\begin{question}\label{intJcurve}
Suppose $(M^{2n}, J)$ is an almost complex $2n$-dimensional manifold, $Z_1$ is an irreducible $J$-holomorphic subvariety and $Z_2$ is a compact connected almost complex submanifold of (real) codimension $2$. If the intersection $Z_1\cap Z_2$ is not one of $Z_i$, is it a $J$-holomorphic subvariety of dimension $\dim_{\mathbb R} Z_1-2$?
\end{question}

The statement is apparently true if $Z_1$ and $Z_2$ intersect transversely, or the intersection is known to be a smooth manifold. When $n=2$, this is true by positivity of intersection of pseudoholomorphic curves \cite{Gr, McD, MW}. Then the next situation after Theorem \ref{Gromov} is when $\dim Z_1=4$, {\it i.e.} the intersection of a pseudoholomorphic $2$-subvariety and a divisor. The following slightly more general result is obtained in \cite{ZCJM}.

\begin{theorem}\label{ICdim4}
Suppose $(M^{2n}, J)$ is an almost complex $2n$-manifold, and  $Z_2$ is a codimension $2$ compact connected almost complex submanifold. Let $(M_1, J_1)$ be a compact connected almost complex $4$-manifold and $u: M_1\rightarrow M$ a pseudoholomorphic map such that $u(M_1)\nsubseteq Z_2$. Then $u^{-1}(Z_2)$ supports a $J_1$-holomorphic $1$-subvariety in $M_1$.
\end{theorem}
A quick corollary of Theorem \ref{ICdim4} which directly answers Question \ref{intJcurve} when $\dim_{\R} Z_1=4$ and suffices for many applications is the following.

\begin{cor}\label{eICdim4}
Suppose $(M^{2n}, J)$ is an almost complex $2n$-dimensional manifold, and $Z_1$, $Z_2$ are compact connected almost complex submanifolds of dimension $4$ and $2n-2$ respectively. Then the intersection $Z_1\cap Z_2$ is either one of $Z_i$, or supports a $J$-holomorphic $1$-subvariety. 
\end{cor}

Notice we do not require our almost complex structures $J$ or $J_1$ to be tamed by a symplectic form. Recall that an almost complex structure $J$ is said to be
tamed by a symplectic form $\omega$ if the bilinear form $\omega(\cdot, J\cdot)$ is positive definite. We say $J$ is tamed if we do not specify such a symplectic form albeit there exists one. An almost complex structure $J$ is compatible with $\omega$ if $J$ is tamed by $\omega$ and $\omega(v, w)=\omega(Jv, Jw)$ for any $v, w\in TM$. We also say such a $J$ is almost K\"ahler if we do not specify a compatible symplectic form. 

In the statement of Theorem \ref{ICdim4}, a set is said to supporting a pseudoholomorphic $1$-subvariety if it is the support $|\Theta|=\cup_{(C_i, m_i)\in \Theta}C_i$ of a pseudoholomorphic $1$-subvariety $\Theta$. In fact, we are also able to determine the homology class of the $J_1$-holomorphic $1$-subvariety in Theorem \ref{ICdim4}.  The homology class $e_{\Theta}=\sum_{(C_i, m_i)\in \Theta} m_i[C_i]$ is calculated by the homology class of the submanifold of $M_1$ that is obtained in a similar manner but using a (smooth) perturbation of $u$ that is transverse to $Z_2$. See \cite{ZCJM} for details.

Notice that the image of $u$ in Theorem \ref{ICdim4} might not be a $J$-holomorphic subvariety of dimension $4$. If $u(M_1)$ is of dimension $0$, then it is a point, and $u(M_1)\cap Z_2=\emptyset=u^{-1}(Z_2)$ since $u(M_1)\nsubseteq Z_2$. If $u(M_1)$ is a $J$-holomorphic $1$-subvariety, since $u(M_1)\nsubseteq Z_2$ and $u(M_1)$ is connected, then $u(M_1)\cap Z_2$ is a collection of finitely many points possibly with multiplicities, {\it i.e.} a $0$-dimensional subvariety. Moreover, the preimage of each such point is a $J_1$-holomorphic $1$-subvariety. If $u(M_1)$ is of dimension $4$, $u(M_1)\cap Z_2$ is the image of $J_1$-holomorphic subvarieties $u^{-1}(Z_2)$. And each irreducible component of $u^{-1}(Z_2)$ is either contracted to a point, or mapped to a $J$-holomorphic curve.

Let us briefly explain the idea of the proof of Theorem \ref{ICdim4}. The strategy is to view the intersecting set as a positive current and then use certain regularity results to show it is indeed a pseudoholomorphic subvariety. Or slightly more precise, we first show that set $A=u^{-1}(Z_2)$ intersects positively with all local $J_1$-holomorphic disks, then we show that $A$ is a $J_1$-holomorphic $1$-subvariety. This basic strategy dates back to \cite{King} at least, where it works in the complex analytic setting. In the pseudoholomorphic situation, this strategy is worked out in \cite{T}. 

It is helpful to first see a way to show a submanifold is almost complex: If $\Sigma \subset \C^2$ is a codimension $2$ submanifold with positive local intersection with all complex lines,  then $\Sigma$ is complex. To show that, we use Taylor's theorem with remainder to represent $\Sigma$ near a point as a graph over its tangent space. If the tangent space is not complex at the given point, then it is easy to write down a complex line which has non-positive intersection number with $\Sigma$ at the point by applying the toy model in the beginning of this section.

However, our set  $A=u^{-1}(Z_2)$ is more complicated, it is certainly not a submanifold in general. Moreover, to formulate such a statement for $A$, we need a notion of intersection number which mimics our definition of multiplicities of a smooth function in the toy model. 
For this sake, Taubes introduces the notion of ``positive cohomology assignment", which plays the role of intersection number of our set $A$ with each local open disk. We assume $(X, J)$ is an almost complex manifold, and $C\subset X$ is merely a subset at this moment. Let $D\subset \mathbb C$ be the standard unit disk. A map $\sigma: D\rightarrow X$ is called {\it admissible} if $C$ intersects the closure of $\sigma(D)$ inside $\sigma(D)$. The following is extracted from section 6.1(a) of \cite{T}.

\begin{definition}\label{PCA}
A positive cohomology assignment to the set $C$ is an assignment of an integer, $I(\sigma)$, to each admissible map $\sigma: D\rightarrow X$. Furthermore, the following criteria have to be met: 
\begin{enumerate}
\item If $\sigma: D\rightarrow X\setminus C$, then $I(\sigma)=0$. 

\item If $\sigma_0, \sigma_1: D\rightarrow X$ are admissible and homotopic via an admissible homotopy (a homotopy $h:[0, 1]\times D\rightarrow X$ where $C$ intersects the closure of Image$(h)$ inside Image$(h)$), then $I(\sigma_0)=I(\sigma_1)$.

\item Let $\sigma: D\rightarrow X$ be admissible and let $\theta: D\rightarrow D$ be a proper, degree $k$ map. Then $I(\sigma\circ \theta)=k\cdot I(\sigma)$.

\item Suppose that $\sigma: D\rightarrow X$ is admissible and that $\sigma^{-1}(C)$ is contained in a disjoint union $\cup_iD_i\subset D$ where each $D_i=\theta_i(D)$ with $\theta_i: D_i\rightarrow D$ being an orientation preserving embedding. Then $I(\sigma)=\sum_iI(\sigma\circ \theta_i)$.

\item If $\sigma: D\rightarrow X$ is admissible and a $J$-holomorphic embedding with $\sigma^{-1}(C)\ne \emptyset$, then $I(\sigma)>0$.
\end{enumerate}
\end{definition}

We can see that these are exactly the same $5$ properties when we define the multiplicity of a continuous function in the beginning of this section.

In this terminology, the simple current argument to show a submanifold is almost complex is generalized by the following result of Taubes \cite{T}.

\begin{theorem}\label{pcaholo}
Let $(X, J)$ be a $4$-dimensional almost complex manifold and let $C\subset X$ be a closed set with finite $2$-dimensional Hausdorff measure and a positive cohomology assignment. Then $C$ supports a compact $J$-holomorphic $1$-subvariety. 
\end{theorem}

By virtue of this, the proof of Theorem \ref{ICdim4} is divided by two parts. The first step is to show $A=u^{-1}(Z_2)$ has finite $2$-dimensional Hausdorff measure. To establish this, I first prove a generalization of unique continuation of pseudoholomorphic curves to higher dimensions. This prevents $A$ from being an open subset of $M_1$, thus reduces our discussion to a small open neighborhood of any point in $M_1$. Then we use a dimension reduction argument. By Gromov's positivity of intersections, a dimension $2$ and a codimension $2$ connected almost complex submanifolds intersect at isolated points positively. This could be used to show the $2$-dimensional Hausdorff dimension of the intersection $A=u^{-1}(Z_2)$ is finite with the help of a local smooth foliation of $J_1$-holomorphic disks on $M_1$. Then the coarea formula would imply the finiteness of the $2$-dimensional Hausdorff measure of $A$.

The second part is to find the positive cohomology assignment when $A$ is considered as a subset in $4$-manifold $M_1$. The idea is, instead of using the set $A$ directly, one assigns the intersection number of the image of our test disk in $M$ with the submanifold $Z_2$ in the ambient manifold $M$. This ends the sketch of the proof of Theorem \ref{ICdim4}.

\subsection{Further discussions}
If we insert Theorem \ref{MicW} into our argument to replace Theorem \ref{Gromov}, we have the following stronger result when the ambient space is of dimension $4$. 

\begin{theorem}\label{ICdim4MW}
Suppose $(M^{4}, J)$ is an almost complex $4$-manifold, and  $Z_2$ is a $J$-holomorphic $1$-subvariety. Let $(M_1, J_1)$ be a closed connected almost complex $4$-manifold and $u: M_1\rightarrow M$ a pseudoholomorphic map such that $u(M_1)\nsubseteq Z_2$. Then $u^{-1}(Z_2)$ supports a $J_1$-holomorphic $1$-subvariety in $M_1$.
\end{theorem}

We certainly expect Theorem \ref{ICdim4} would hold in this generality for higher dimensional ambient spaces as well.

To say a bit more on the proof, the overdeterminedness of the equations satisfied by almost complex submanifolds of complex dimension $>1$ (though it is already the case in the complex setting) was one of the main reasons preventing people from studying these submanifolds. The other reason is that there are no classical technical tools since we are out of the boundary of elliptic equations. The above described foliation-by-disks technique is essentially ``foliating" an overdetermined system by elliptic systems, where each leaf has positivity of intersection (by Gromov's positivity of intersection of a curve and a divisor, Theorem \ref{Gromov}, or Micallef-White's Theorem \ref{MicW}). On the other hand, it is possible to choose some natural parametrizing spaces of almost complex structures to make it an elliptic system as a whole. The intersection theory of each ``fiber" is of fundamental importance.

This technique of foliating a neighborhood by $J$-holomorphic disks could be generalized to higher ambient dimensions ({\it e.g.} Lemma 3.10 of \cite{ZCJM}). It implies that the first half of Theorem \ref{ICdim4} still holds when $\dim M_1>4$. Precisely, if we replace $(M_1, J_1)$ in the statement by a compact connected almost complex manifold of dimension $2k$, then $u^{-1}(Z_2)\subset M_1$ is a closed set with finite $(2k-2)$-dimensional Hausdorff measure and a positive cohomology assignment. 
Hence, it is reduced to show the following generalization of Theorem \ref{pcaholo}.
\begin{question}\label{pcaholoco2}
Let $(X, J)$ be a closed $2k$-dimensional almost complex manifold and let $C\subset X$ be a closed set with finite $(2k-2)$-dimensional Hausdorff measure and a positive cohomology assignment. Does $C$ support a $J$-holomorphic subvariety of complex dimension $k-1$?
\end{question}
 
 So far, our intersection theory in the almost complex category always has a divisor as one side of the intersecting object. To go beyond this codimension two restriction, we will encounter the excess intersection, as mentioned above. The relevant situation of lowest dimension to study is the intersection of two $4$-dimensional almost complex submanifolds in a $8$-dimensional almost complex manifold. This is a case where Taubes' Theorem \ref{pcaholo} could still be applied. The difficulty here is to find a positive cohomology assignment. One should be able to establish the excess intersection formula \cite{Ful} at the same time. Such a formula is known when the intersection is clean (see {\it e.g.} \cite{CG}). 

Lastly, we remark that the intersection theory of almost complex submanifolds could be used to study symplectic birational geometry, which is initiated by Yongbin Ruan in 1990s to generalize algebraic birational geometry to symplectic category. Some very preliminary applications to this subject could be found in \cite{ZCJM}. It is believed to be useful in realizing Ruan's idea of symplectic supporting divisors which transplant ideas of big and nef divisors in algebraic geometry. 

\subsection{Orbifolds}
We can also generalize our results to (smooth) orbifolds. We follow the notation of \cite{WC} where an introduction to the orbifold theory on symplectic geometry and pseudoholomorphic curves is provided. An $n$-dimensional orbifold structure on a paracompact Hausdorff space $X$ is given by an atlas of local charts $\mathbb U$, where $\mathbb U=\{U_i|i\in I\}$ is an open cover of $X$ such that for any $p\in U_i\cap U_j, U_i, U_j\in \mathbb U$, there is a $U_k\in \mathbb U$ with $p\in U_k\subset U_i\cap U_j$. Moreover, 
 \begin{itemize}
 \item for each $U_i\in \mathbb U$, there exists a triple $(V_i, G_i, \pi_i)$, called a local uniformizing system, where $V_i$ is an $n$-dimensional smooth manifold, $G_i$ is a finite group acting smoothly and effectively on $V_i$, and $\pi_i: V_i\rightarrow U_i$ is a continuous map inducing a homeomorphism $U_i\cong V_i/G_i$;
\item for each pair $U_i, U_j\in \mathbb U$ with $U_i\subset U_j$, there is a set $Int(U_i, U_j)=\{(\phi, \lambda)\}$, whose elements are called injections, where $\phi: V_i\rightarrow V_j$ is a smooth open embedding and $\lambda: G_i\rightarrow G_j$ is an injective homomorphism, such that $\phi$ is $\lambda$-equivariant and satisfies $\pi_i=\pi_j\circ\phi$, and $G_i\times G_j$ acts transitively on $Inj(U_i, U_j)$ by $$(g, g')\cdot (\phi, \lambda)=(g'\circ\phi\circ g^{-1}, Ad(g')\circ\lambda\circ Ad(g^{-1});$$ 
\item the injections are closed under composition for any $U_i, U_j, U_k\in \mathbb U$ with $U_i\subset U_j\subset U_k$. 
\end{itemize}

The almost complex structure and symplectic structure on orbifolds are natural extensions of the ones for manifolds. For example, an almost complex structure on an orbifold $X$ is an endomorphism $J: TX\rightarrow TX$ with $J^2=-1$, which is given by a family of endomorphisms $\{J_i: TV_i\rightarrow TV_i\}$ with $J_i^2=-1$ for each local uniformizing system $(V_i, G_i, \pi_i)$, which is equivariant under the local group actions and compatible with respect to the injections. Pseudoholomorphic maps considered in this category are maps from a compact almost complex orbifold $X$ to another almost complex orbifold $M$, $f: X\rightarrow M$, such that if $f(z)=p$ and $(D_z, G_z)$, $(V_p, G_p)$ be local uniformizing systems at $z, p$ respectively, then $f$ determines a germ of pairs $(\hat h_z, \rho_z)$, where $\rho_z: G_z\rightarrow G_p$ is an injective homomorphism and $\hat h_z: D_z\rightarrow V_p$ is $\rho_z$-equivariant and $J$-holomorphic. Then our pseudoholomorphic subvarieties could be defined correspondingly as the positively weighted image of pseudoholomorphic maps as well. 

The adjunction formula and the intersection formula, in particular, the positivity of intersection for pseudoholomorphic curves, are extended in orbifold setting, as can be found in \cite{WC}. The results in \cite{T} are also extended to orbifolds in \cite{WCs}. In particular, the proof of Taubes' Theorem \ref{pcaholo} could be extended to the orbifold setting without any essential change. As rest of our argument for Theorem \ref{ICdim4} (and Theorem \ref{ICdim4MW}) is local, these could be worked on local uniformizing systems of orbifolds, hence the results in this section could be extended to orbifolds setting. 

\begin{theorem}\label{ICdim4orb}
Suppose $(M^{2n}, J)$ is an almost complex $2n$-orbifold, and  $Z_2$ is a codimension $2$ compact connected almost complex sub-orbifold. Let $(M_1, J_1)$ be a compact connected almost complex $4$-orbifold and $u: M_1\rightarrow M$ a pseudoholomorphic map such that $u(M_1)\nsubseteq Z_2$. Then $u^{-1}(Z_2)$ supports a $J_1$-holomorphic $1$-subvariety in $M_1$.

When $n=2$, we can take $Z_2$ to be a $J$-holomorphic $1$-subvariety in the assumption. 
\end{theorem}

\section{Maps between almost complex manifolds and birational morphisms}
The results in this section are built on the intersection theory of almost complex submanifolds. We start by motivating these through our smooth to almost complex philosophy. 

For a smooth map $u: X\rightarrow M$, the {\it singularity subset} $S_u$ of $u$ is the set of points $p\in X$ such that the differential $du_p: T_pX\rightarrow T_{u(p)}M$ is not of full rank. There are two classical results on singularities of differential maps.
 
\begin{theorem}[Sard's theorem] 
Suppose $u: X\rightarrow M$ is a $C^k$ map. If $k\ge \max\{\dim X-\dim M+1, 1\}$, then the image $u(\mathcal S_u)$ has measure zero as a subset of $M$.
\end{theorem}

A more general form of  Sard's theorem \cite{Sa2} says that the image of the set of points $x\in X$ such that the differential $du_x$ has rank strictly less than $r$ has $r$-dimensional Hausdorff measure zero, under the same regularity assumption on $u$.

Another result is Thom's and Boardman's fundamental work \cite{Th, Boa}. Although it would need more space to phrase its precise form, the core statement could be summarized as the following. 

\begin{theorem}[Thom and Boardman] 
For generic smooth maps between smooth manifolds, the singularity subsets with given degeneracy data are  submanifolds of the domain. 
\end{theorem}

The two theorems together describe a nice structure for the singularities of a generic differential map. We look for the corresponding results of singularities for an arbitrary pseudoholomorphic map between two almost complex manifolds. We start from the simplest case, {\it i.e.} when $X=\Sigma$ is a Riemann surface. In this situation, we have the well known fact that for non-constant pseudoholomorphic curves $u: \Sigma\rightarrow (M, J)$, $\mathcal S_u$ is a set of isolated points.

In \cite{ZCJM}, we obtain the following finer description of the structure of pseudoholomorphic maps between closed almost complex $4$-manifolds. The first part corresponds to Thom and Boardman's theorem and the second part is a finer version of Sard's theorem. 

\begin{theorem}\label{sing4to4}
Let $u: (X, J)\rightarrow (M, J_M)$ be a somewhere immersed pseudoholomorphic map between closed connected almost complex $4$-manifolds. Then 
\begin{itemize}
\item the singularity subset of $u$ supports a $J$-holomorphic $1$-subvariety; 
\item other than finitely many points $x\in M$, where $u^{-1}(x)$ is the union of a $J$-holomorphic $1$-subvariety and finitely many points, the preimage of each point is a set of finitely many points.
\end{itemize}
\end{theorem}

We would like to explain how our intersection theory is applied to show the singularity subset $\mathcal S_u$ supports a $J$-holomorphic $1$-subvariety. The idea is to apply the intersection theory Theorem \ref{ICdim4} to the determinant bundle $\mathcal L$ of the $1$-jets bundle of pseudoholomorphic mappings $J^1(X, M)$. The manifold $J^1(X, M)$ is the  total space of the complex vector  bundle $\mathcal E$ over $X\times M$, whose fiber is the complex vector space of all complex linear maps $L: T_xX\rightarrow T_mM$ regarding the almost complex structures $J|_x$ and $J_M|_m$. Its determinant bundle $\mathcal L$ is a complex line bundle over $X\times M$ obtained by replacing fibers of $\mathcal E$ by $\det L: \Lambda_{\mathbb C}^2T_xX \rightarrow \Lambda^2_{\mathbb C}T_mM$. The bundle $\mathcal L$ has a canonical almost complex structure $\mathcal J$ induced from any of the canonical almost complex structures on $J^1(X, M)$ as in \cite{Gau} or \cite{LS}. Moreover, the map $u$ induces a pseudpholomorphic map  $u_{\mathcal L}(x)=(x, u(x), \det (du_x)_{\mathbb C})$ from $(X, J)$ to $(\mathcal L, \mathcal J)$. 

Under this description, the singularity subset $\mathcal S_u=u_{\mathcal L}^{-1}(X\times M\times \{0\})$. Apply Theorem \ref{ICdim4} to the codimension two almost complex submanifold $X\times M\times \{0\}$ in the ambient space  $(\mathcal L, \mathcal J)$ and because $u$ is somewhere immersed we have $u_{\mathcal L}(X)\nsubseteq Z_2$.  Hence, the singularity subset $\mathcal S_u$ supports a $J$-holomorphic $1$-subvariety.  The second part of Theorem \ref{sing4to4} is built on the first part plus a careful discussion on a stratification of $X$.

A closer study at degree one pseudoholomorphic maps between almost complex $4$-manifolds shows that they are eventually birational morphisms in pseudoholomorphic category. First, Zariski's main theorem still holds for pseudoholomorphic maps. Moreover, we have a very concrete description of the exceptional set. It is summarized in the following.

\begin{theorem}\label{introphblowdown}
Let $u: (X, J) \rightarrow (M, J_M)$ be a degree one pseudoholomorphic map between closed connected almost complex $4$-manifolds such that $J$ is almost K\"ahler. Then there exists a subset $M_1\subset M$, consisting of finitely many points, with the following significance:
\begin{enumerate}
\item The restriction $u|_{X\setminus u^{-1}(M_1)}$ is a diffeomorphism. 
\item At each point of $M_1$, the preimage is an exceptional curve of the first kind.  
\item $X\cong M\#k\overline{\mathbb CP^2}$ diffeomorphically, where $k$ is the number of irreducible components of the $J$-holomorphic $1$-subvariety $u^{-1}(M_1)$.
\end{enumerate}
\end{theorem}

The degree one condition guarantees there are no branchings described in part two of Theorem \ref{sing4to4}. Moreover, a delicate study can find the subset $M_1\subset M$ where the preimage of any point in it is a connected $J$-holomorphic $1$-subvariety, which should be understood as a version of Zariski's main theorem in our setting.

Roughly, a connected $J$-holomorphic $1$-subvariety is called an exceptional curve of the first kind if its configuration is equivalent to the empty set through topological blowdowns.  In particular, it is a connected $J$-holomorphic $1$-subvariety whose irreducible components are rational curves and the dual graph is a tree. This latter statement could be shown by a topological argument relating the neighborhoods of a point in $M_1$ and its preimage.

To complete the second part, we need to establish Grauert's criterion for exceptional set, {\it i.e.} the intersection matrix of the irreducible components of the exceptional set is negative definite. We will use a symplectic argument. Let $Y$ be a $3$-dimensional separating submanifold of a symplectic $4$-manifold $(X, \omega_X)$, suppose $Y$ admits a fixed point free circle action whose orbits lie in the null direction of $\omega_X|_Y$. Then we let $X^-$ be the piece for which $Y$ is the $\omega_X$-convex boundary and $X^+$ the other piece. 

For any point $m\in M_1$, we take an open ball neighborhood $\mathcal N_m$ whose boundary is $J_M$-convex and $\mathcal N_m\cap M_1=\{m\}$. We can choose $\mathcal N_m$ such that it is contained in a neighborhood of $m$ such that there exists a symplectic form compatible with $J_M$. The induced contact structure on $\partial \mathcal N_m$ is the unique tight contact structure on $S^3$. This is our $Y$ whose fixed point free $S^1$-action is induced by Reeb orbits. Hence, $\mathcal N_m$ could be capped (by a concave neighborhood of $+1$-sphere) to a symplectic $\mathbb CP^2$. This is our $X_2$. 

We take $X_1$ to be our $X$ with the $J$-compatible symplectic form $\omega_X$. The preimage $u^{-1}(\partial\mathcal N_m)$ is diffeomorphic to $S^3$. Moreover, since $u$ is pseudoholomorphic and $u|_{u^{-1}(\partial\mathcal N_m)}$ is a diffeomorphism, the $J$-lines provide a contact structure on $u^{-1}(\partial\mathcal N_m)$ which is contactomorphic to the one on $\partial \mathcal N_m$ induced by $J_M$-lines. Hence, we can apply McCarthy-Wolfson's gluing theorem \cite{McW} to obtain a symplectic manifold $\tilde X=X_1^-\cup_Y X_2^+$. This is why the almost K\"ahler condition comes into the statement, although I think this assumption should be removable. 

Since $X_2^+$ contains a symplectic sphere $S$ of self-intersection $1$, the symplectic $4$-manifold $\tilde X$ is diffeomorphic to $\mathbb CP^2\# k\overline{\mathbb CP^2}$, {\it i.e.} a rational $4$-manifold \cite{McD90}. In particular, $b^+(\tilde X)=1$. Since $C_m=\pi^{-1}(m)$ is disjoint from $S$, we have $Q_{C_m\cup S}=Q_{\C_m}\oplus (1)$ as (a sub matrix of) the intersection matrix of $\tilde X$. It implies the matrix $Q_{C_m}$ corresponding to $C_m$ is negative definite.  Finally, to show it is an exceptional curve of first kind, we apply the classification of these type of curve configurations in \cite{LM, Sha}.

\subsection{Some further discussions}
First, the resolution of Question \ref{pcaholoco2} would generalize the statement of Theorem \ref{sing4to4} to somewhere immersed pseudoholomorphic maps between closed connected almost complex $2n$-manifolds, where $1$-subvarieties would be replaced by $(n-1)$-subvarieties correspondingly. However, the argument of Theorem \ref{introphblowdown} is not extended to higher dimensions as it uses some peculiar properties in dimension $4$, {\it e.g.} any rational surface has $b^+=1$. Meanwhile, to study the singularity subset of a pseudoholomorphic maps between manifolds of different dimensions, we need the intersection theory of excess intersection. 

These higher dimensional generalizations will be very important in Ruan's symplectic birational geometry. As a crucial step in the program, we need to classify symplectic birational morphisms. Li and Ruan had thought certain class of degree one maps might serve such a role. As observed in \cite{Zadv}, a plain differentiable degree one map will not preserve the birational class in general. Thus these degree one maps must regard the symplectic structures. A potential candidate of a symplectic birational morphism is a degree one $(J_1, J_2)$-holomorphic map $u: (X, \omega_X)\rightarrow (M, \omega_M)$ where $J_1$ and $J_2$ are tamed by $\omega_X$ and $\omega_M$. Hence, any work on degree one pseudoholomorphic maps would lead to applications in (this version of) symplectic birational morphisms. 

In the definition of (irreducible) pseudoholomorphic subvarieties, we require the maps to be somewhere immersed. It is natural to ask whether the image of an arbitrary pseudoholomorphic map $f: (X, J_X)\rightarrow (M, J)$ from a closed almost complex manifold $X$ is a $J$-holomorphic subvariety. By a classical result of Remmert \cite{Rem}, this is true when both $J_X$ and $J$ are integrable. In fact, only the integrability of $J$ is needed, as observed in \cite{CZ2}. When there is no integrability in hand, Micallef and I are able to show the following. 

\begin{theorem}
Let  $f: X\rightarrow M$ be a pseudoholomorphic map from a connected closed almost complex manifold $(X, J_X)$ to a connected almost complex manifold $(M, J)$. If $\min\{\dim X, \dim M\} \le 4$, then $f(X)$ is an irreducible $J$-holomorphic subvariety.    
\end{theorem}

It is also important to study pseudoholomorphic rational maps $f: X\dashrightarrow Y$, {\it i.e.} a proper pseudoholomorphic map $\phi: X\setminus B\rightarrow Y$ where $B$ has Hausdorff codimension at least two, for example, when it is a pseudoholomorphic subvariety. 
For example, it is crucial in obtaining higher dimensional Hartogs extension in almost complex setting and studying the properties of almost complex Kodaira dimensions (see Section \ref{aKod} and \cite{CZ}). Even when $Y$ is complex, this is not totally understood, however we know the image $\phi(X\setminus B)$ has the structure of analytic subvariety locally near any point in the image following from \cite{CZ2}. By abuse of notation, we call the image an open analytic subvariety. It is fundamental to know the following. 

\begin{question}
Are the closure of $f(X\setminus B)$ and complement of its image subvarieties in $Y$?
\end{question}

Finally, the discussions regarding the singular subset of a pseudoholomorphic map between two manifolds could be extended to degeneracy loci of maps between vector bundles in a strait-forward manner. For any complex vector bundles $E$ and $F$ over an almost complex manifold $X$ endowed with bundle almost complex structures, and a pseudoholomorphic bundle morphism $u$ between them,  we can define the degeneracy locus of rank $r$ of $u$ to be $D_r(u)=\{x\in X| \hbox{rank} u(x)\le r\}$ for any integer $r<\min\{\hbox{rank} E, \hbox{rank}  F\}$.

In particular, the proof of Theorem \ref{sing4to4} leads to the following without any extra difficulty. 

\begin{theorem}
Let $ (E, \mathcal J_E)$ and $(F, \mathcal J_F)$ be vector bundles over closed connected almost complex $4$-manifolds $(X, J)$ with $\hbox{rank}E=\hbox{rank} F=r$ and bundle almost complex structures $\mathcal J_E$ and $\mathcal J_F$. Let $u: (E, \mathcal J_E)\rightarrow (F, \mathcal J_F)$ be a  pseudoholomorphic bundle map. 
Then the degeneracy locus $D_{r-1}(u)$ supports a union of $J$-holomorphic a $1$-subvariety and a $0$-subvariety in $X$.
\end{theorem}

Finally, the argument for Theorems \ref{sing4to4} and \ref{introphblowdown} could be applied to state theorems for pseudoholomorphic maps $u: X\rightarrow M$ from an almost complex $4$-manifold $X$ to almost complex $4$-orbifold $M$, which could be thought as resolution of almost complex orbifold singularities. The statement of Theorem \ref{sing4to4} does not need essential modification. For the corresponding version of Theorem \ref{introphblowdown}, the degree one condition ensures there are only finitely many orbifold points in $M$. Moreover, our Grauert criterion still holds and the preimage of these orbifold points under the map $u$ are equivalent to curve configurations of type (N) listed in \cite{LM} in the sense of topological blow-ups. In other words, we have

\begin{theorem}\label{resorb}
Let $u: (X, J) \rightarrow (M, J_M)$ be a degree one pseudoholomorphic map from a closed connected almost complex $4$-manifold $X$ to a closed connected almost complex $4$-orbifold $M$ such that $J$ is almost K\"ahler. Then there exists a subset $M_1\subset M$, consisting of finitely many points, with the following significance:
\begin{enumerate}
\item The restriction $u|_{X\setminus u^{-1}(M_1)}$ is a diffeomorphism. 
\item At each smooth point in $M_1$, the preimage is an exceptional curve of the first kind.  
\item  At each orbifold point in $M_1$, the preimage is equivalent to one of the configuration of type (N). In particular, each irreducible component is a smooth rational curve and there are no cycles enclosed by irreducible components.
\end{enumerate}
\end{theorem}

\section{From harmonic forms to $J$-anti-invariant forms}
The third set of results motivated from our smooth to almost complex philosophy is on the study of zero locus of $J$-anti-invariant forms \cite{BonZ}.

 For any Riemannian metric $g$ on a $4$-manifold, we have the well-known self-dual, anti-self-dual splitting of the bundle of real $2$-forms, 
\begin{equation} \label{formtypeg}
\Lambda^2=\Lambda_g^+\oplus \Lambda_g^-,
\end{equation}
since the Hodge star $*_g$ is an involution on $\Lambda^2$. Each $\Lambda_g^{\pm}$ is a rank $3$ real bundle. A section of $\Lambda_g^{\pm}$ is an (anti)-self-dual $2$-form, namely $*_g\alpha=\pm\alpha$. It is clear that a closed (anti)-self-dual $2$-form is harmonic. 

By transversality theorem, the zero set of a generic section of a  rank $3$ bundle over $4$-manifold is a $1$-manifold. This could be applied to the bundles $\Lambda_g^{\pm}$. Moreover, if we require the section to be harmonic when viewed as $2$-form, there is a well known folklore theorem around  the 1980s.

\begin{theorem}
For a generic Riemannian metric on a closed $4$-manifold, the zero set of a non-trivial self-dual harmonic $2$-form is a finite number of embedded circles.
\end{theorem}
This is the starting point of Taubes' program, {\it e.g.} \cite{Tsd}, to generalize the identification of Seiberg-Witten invariants and Gromov invariants for symplectic $4$-manifolds to general compact oriented $4$-manifolds. The form is symplectic off these circles, and they appear as the boundary of the pseudoholomorphic curves in the correspondence.

For the corresponding almost complex version, we start with a less well known decomposition of the $2$-form bundle $\Lambda^2$. For almost complex manifold $(M, J)$, $J$ acts on $\Lambda^2$ 
as an
 involution, $\alpha(\cdot, \cdot) \rightarrow \alpha(J\cdot,
J\cdot)$. This involution induces the splitting into bundles of $J$-invariant and
$J$-anti-invariant 2-forms
\begin{equation*} \label{formtypeJ}
\Lambda^2=\Lambda_J^+\oplus \Lambda_J^-
\end{equation*}
corresponding to the eigenspaces of eigenvalues $\pm 1$ respectively. The sections of $\Lambda^{\pm}$ are called  $J$-invariant and
$J$-anti-invariant 2-forms respectively. The bundle  $\Lambda^-_J$ inherits an almost complex structure, still denoted by $J$, from $J\alpha(X, Y)=-\alpha(JX, Y)$. $\Lambda_J^-$ is a rank $2$ bundle, so the expected dimension of the zero set of a section over $M^4$ is $2$.

 When $g$ is compatible with $J$, {\it i.e.} $g(Ju, Jv)=g(u, v)$, we have $\Lambda_J^-\subset \Lambda_g^+$. Hence any closed $J$-anti-invariant $2$-form is $g$-harmonic. In fact, we can see for every fiber,  $\Lambda_J^-$ is spanned by $\Re ((dx_1+idx_2)\wedge (dx_3+idx_4)), \Im ((dx_1+idx_2)\wedge (dx_3+idx_4))$, and $\Lambda_g^+$ is spanned by these forms and an additional $dx_1\wedge dx_2+dx_3\wedge dx_4$.

The complex line bundle $\Lambda_J^-$ can be viewed as a natural generalization of the canonical bundle of a complex manifold in the complex setting.  On a complex surface $(M, J)$, if $\alpha$ is a closed $J$-anti-invariant $2$-form, then $J\alpha$ is also closed and $\alpha+iJ\alpha$ is a holomorphic $(2, 0)$ form. Hence the zero set $\alpha^{-1}(0)$ is a canonical divisor of $(M, J)$, {\it e.g.} by the Poincar\'e-Lelong theorem.  

Hence, in terms of the results in smooth and complex categories, our smooth to almost complex philosophy suggests that the zero set of a closed $J$-anti-invariant $2$-form is a $J$-holomorphic curve when $J$ is almost complex. This speculation is a question raised in the proceeding of fifth International Congress of Chinese Mathematicians \cite{DLZiccm}, which is answered affirmatively in \cite{BonZ}.

\begin{theorem}\label{4ZJhol}
Suppose $(M, J)$ is a compact connected almost complex $4$-manifold and $\alpha$ is a non-trivial closed $J$-anti-invariant $2$-form. Then the zero set $Z$ of $\alpha$ supports a $J$-holomorphic $1$-subvariety $\Theta_{\alpha}$ in the canonical class $K_J$.
\end{theorem}

The general scheme to prove Theorem \ref{4ZJhol} is similar to what is used to prove Theorem \ref{ICdim4}. We first show that $Z$ has finite $2$-dimensional Hausdorff measure. The idea is to foliate neighborhoods of points in $Z$ by $J$-holomorphic disks. Applying a dimension reduction argument with the help of a unique continuation result, we are able to reduce our study to the intersection of $Z$ with $J$-holomorphic disks. Then a key difference from Theorem \ref{ICdim4} is that one need a result corresponding to Gromov's positivity of intersection for dimension $2$ and codimension $2$ almost complex submanifolds (Theorem \ref{Gromov}). Our strategy is to choose a good local frame around any point $x\in M$ such that a closed $J$-anti-invariant $2$-form, viewed as a section of $\Lambda_J^-$, could be written as a holomorphic function locally along each disk which is used to foliate a neighborhood. This is Lemma 2.3 in \cite{BonZ}.

The second step is to find a ``positive cohomology assignment" for $Z$. The strategy is to view $J$-anti-invariant $2$-forms as sections of the bundle $\Lambda_J^-$. Now a $J$-anti-invariant form $\alpha$ defines a $4$-dimensional submanifold $\Gamma_{\alpha}$ in the total space of $\Lambda_J^-$ whose intersection with $M$, as submanifolds of $\Lambda_J^-$, is the zero set of the form. Given a disk in $M$, whose boundary does not intersect $\Gamma_{\alpha}$, we can compose with a section and perturb it to obtain an admissible disk $\sigma ' : D \rightarrow \Lambda_J^-$ which intersects $M$ transversely. Then the oriented intersection number of $\sigma'$ with the zero section defines a positive cohomology assignment. 

Theorem \ref{4ZJhol} implies that for any tamed $J$ on $T^4$ or a K3 surface, the zero set of any closed $J$-anti-invariant $2$-form is empty. We would like to know wether there are examples of  closed $J$-anti-invariant $2$-forms for an almost complex structure $J$ on $T^4$ or K3 whose zero sets are non-empty. The examples constructed in \cite{DLZ2} on these manifolds all have empty zero sets. 

Theorem \ref{4ZJhol} could be extended to the sections of bundle $\Lambda_{\R}^{n, 0}$ of real parts of $(n, 0)$ forms, which has a natural complex line bundle structure induced by the almost complex structure on $M$. The space of its sections is denoted $\Omega_{\R}^{n, 0}$. We are able to show that the zero set of a non-trivial closed form in $\Omega_{\R}^{n, 0}$ supports a $J$-holomorphic $(n-1)$-subvariety up to Question \ref{intJcurve}. The key to establish this result is again a good choice of local frame for the bundle $\Lambda_{\R}^{n, 0}$. 

\subsection{Relation with Taubes' SW=Gr}
As we mentioned,  Theorem \ref{4ZJhol} is related to Tabues's SW=Gr program for compact oriented $4$-manifolds. To start, we need the notion of near-symplectic forms. A closed $2$-form $\omega$ is called near-symplectic if at each point $x$, either $(\omega\wedge\omega)(x)>0$ or  $\omega(x)=0$ and the derivative $(\nabla \omega)(x): T_xX\rightarrow \Lambda^2T_x^*X$ has rank $3$. It follows from the definition that the zero set $Z$ of a near-symplectic form is a disjoint union of embedded circles.
 If $\omega$ is near symplectic, then there is a Riemannian metric $g$ such that $\omega$ is a self-dual harmonic form with respect to $\omega$. Conversely, if $X$ is compact and $b^+>0$, then for a generic Riemannian metric there is a closed self-dual $2$-form which is near-symplectic. This symplectic form and the metric define a compatible almost complex structure $J$ on $X\setminus Z$. 
 
 Taubes \cite{Tsd} proves that if $X$ has a non-zero Seiberg-Witten invariant then there exists a $J$-holomorphic subvariety in $X\setminus Z$ homologically bounding $Z$ in the sense that it has intersection number $1$ with every linking $2$-sphere of $Z$. This result could be stated for the completion of $X\setminus Z$ by attaching symplectization ends and $J$-holomorphic curves with certain asymptotic conditions \cite{Thwz}.  Gromov type invariants are defined in this latter setting  \cite{Ge}. As the space of $\omega$-tamed almost complex structures is contractible, it follows from SFT compactness theorem \cite{EGH} that the above existence of pseudoholomorphic subvariety in $X\setminus Z$ holds for an arbitrary almost complex structure on $X\setminus Z$ tamed by $\omega$.

It is interesting to comparing the pseudoholomorphic subvarieties in the near-symplectic setting and the one from Theorem \ref{4ZJhol}. We choose a smooth family of closed $2$-forms $\alpha_t$, $ t\in [0,1]$, such that $\alpha_0$ is a $J$-anti-invariant form and any other $\alpha_t$ is near symplectic. If we choose a smooth family of almost complex structures $J_t$ defined on $X\setminus Z_t$ when $t\ne 0$ and on $X$ when $t=0$, such that $J_0=J$ and $J_t$ is tamed by $\alpha_t|_{X\setminus Z_t}$. Suppose the Seiberg-Witten invariant of the canonical class is non-trivial. 
By Taubes' result, there are $J_t$-holomorphic subvarieties bounding $Z_t$ when $t\ne 0$. Under this assumption, we should be able to choose a family of such pseudoholomorphic subvarieties $C_t$ such that $Z_0$ is the limit of them in the Gromov-Hausdorff sense. 

This limiting process is related to the local modification of the boundary circles as in the Luttinger-Simpson theorem. 
For example, there are an even number of untwisted zero circles in $Z_t$ since $X$ admits an almost complex structure by the observation of Gompf that the number of these circles has the same parity as $b_1+b^++1$. 

The types of the zero circles \cite{Ge} are recalled in the following. For any oriented circle in $Z_t$, we denote $z$ to be its unit-length tangent vector. Its normal bundle $N$ is split as $L^+\oplus L^-$ where the quadratic form $N\rightarrow \underline{\R}$, $v\mapsto \langle \iota(z)\nabla_v\omega, v\rangle$ is positive and negative definite respectively. A component of $Z_t$ is called untwisted (resp. twisted) if the line bundle $L^-$ is trivial (resp. non-trivial). 

It would be very interesting to see how the boundary circles ({\it i.e.} the zero circles $Z_t$) and the pseudoholomorphic subvarieties $C_t$ change as $t$ varies. 
In general, the number of boundary circles will change and  thus the topological type of $C_t$ will  change. For instance, it is possible that two untwisted simple boundary circles come together and die, and at the same time the genus of $C_t$ increases by one. As $t$ goes to zero, the number of boundary circles would decrease to zero and we would finally get a closed pseudoholomorphic subvariety $C_0$. It is also possible that Lagrangian Whitney disks \cite{Ticm} play a role in the deformation.

On the other hand, we speculate that the existence of a non-trivial closed $J$-anti-invariant $2$-form implies the existence of an unbounded sequence $\{r_n\}\subset [1, \infty)$ such that the  perturbed Seiberg-Witten equations (2.9) in \cite{Tsd} has solutions, with the spin$^c$ structure whose positive spinor bundle $S^+=T_J^{2,0}M\oplus \underline{\mathbb C}$ and any self-dual harmonic $2$-form $\omega$. Notice we cannot expect the corresponding Seiberg-Witten invariant to be non-trivial. In fact, any non-K\"ahler proper elliptic surface without singular and multiple fibers has vanishing Seiberg-Witten invariant (see Example 1 of \cite{Biq}) but they have $p_g>0$.

\section{Non-compact manifolds}
Although our intersection theory  results are stated only for compact $M_1$ and $Z_2$ so far, the discussion could be extended to non-compact setting. First of all, we should define pseudoholomorphic subvarieties in a general, possibly non-compact, almost complex manifold. 

 An (open) irreducible $J$-holomorphic $n$-subvariety in $M$ is the image of a somewhere immersed  pseudoholomorphic map $\phi: X\rightarrow M$ from a connected smooth almost complex $2n$-manifold $X$. A subset $V\subset M$ is called a $J$-holomorphic $n$-subvariety if it is a countable collection of pairs $\{(V_i, m_i)\}$ where each $V_i$ is an irreducible $J$-holomorphic $n$-subvariety, which is the image of $\phi_i: X_i\rightarrow M$, such that $\cup_iV_i=V$, and moreover there is a set $\Lambda_0$ with countably many connected components $\Lambda^j$ such that any point $x\in\Lambda^j\subset \Lambda_0$ is not in $\overline{\Lambda_0\setminus\Lambda^j}$ and  $\phi_i$ is an embedding outside $\Lambda_0$. 
 
When $M$ is an oriented manifolds with symplectic form $\omega$ and $J$ is tamed by $\omega$, a $J$-holomorphic $k$-subvariety is called  finite energy if furthermore we have  the sum of the integrals $\phi_i^*(\omega^k)$ over $X_i$ is finite.

When $k=1$, the above definition is essentially what can be found in \cite{Tsd} if in addition the subset $V$ is closed. 
Also mentioned there, the crucial Theorem \ref{pcaholo} above could be also stated in this setting.

\begin{theorem}\label{pcaholoopen}
Let $(X, J)$ be a $4$-dimensional almost complex manifold. 
Suppose that $C\subset X$ is a closed set with the following properties:
\begin{itemize}
\item The restriction of $C$ to any open $X'\subset X$ with compact closure has finite $2$-dimensional Hausdorff measure. 
\item $C$ has a positive cohomology assignment. 
\end{itemize}

Then $C$ supports a $J$-holomorphic $1$-subvariety. 
\end{theorem}

Here a positive cohomology assignment is exactly what is in Definition \ref{PCA}, without any change. Moreover, the two parts of the proof for Theorem \ref{ICdim4} can be also applied to show the two bullets in Theorem \ref{pcaholoopen} for the set $A=u^{-1}(Z_2)$ in the non-compact setting if we can show $A$ is actually closed in a subset of $M_1$. First our statement in this setting is the following. 

\begin{theorem}\label{ICdim4open}
Suppose $(M, J)$ is an almost complex $2n$-manifold, 
and  $Z_2$ is a codimension $2$ embedded almost complex submanifold. 
Let $(M_1, J_1)$ be a  connected almost complex $4$-manifold and $u: M_1\rightarrow M$ a pseudoholomorphic map such that $u(M_1)\nsubseteq Z_2$. Then $u^{-1}(Z_2)$ supports a $J_1$-holomorphic $1$-subvariety in $M_1$.

If $J$ has a compatible symplectic form $\omega$ and $Z_2$ has finite energy with respect to $\omega$, then $u^{-1}(Z_2)$ is a finite energy $J_1$-holomorphic $1$-subvariety.
\end{theorem}

As said above, the closedness of $A=u^{-1}(Z_2)$ is very important when applying Theorem \ref{pcaholoopen}. For example, it is crucial for the invariance under homotopy property of a positive cohomology assignment.
 To achieve this, we choose an open subset $U\subset M$ such that $Z_2\subset U$ and $\overline{Z_2}\setminus Z_2\subset M\setminus U$. Then we let $U_1=u^{-1}(U)$ and restrict $u$ on $U_1$. Evidently, $A$ is closed in $U_1$ and Theorem \ref{pcaholoopen} would lead to the $J_1$-holomorphicity of $u|_{U_1}^{-1}(Z_2)$ which is $u^{-1}(Z_2)$.

In dimension $4$, we also have the corresponding version of Theorem \ref{ICdim4MW}.

\begin{theorem}\label{ICdim4MWopen}
Suppose $(M^{4}, J)$ is an almost complex $4$-manifold, and  $Z_2$ is a $J$-holomorphic $1$-subvariety. Let $(M_1, J_1)$ be a connected almost complex $4$-manifold and $u: M_1\rightarrow M$ a pseudoholomorphic map such that $u(M_1)\nsubseteq Z_2$. Then $u^{-1}(Z_2)$ supports a $J_1$-holomorphic $1$-subvariety in $M_1$.
\end{theorem}

We remark that the definition of pseudoholomorphic subvarieties and the intersection results could also be stated for non-compact orbifolds, as in Theorem \ref{ICdim4orb} and the discussions above that. 

Similarly, we have the corresponding version of Theorem \ref{4ZJhol}.

\begin{theorem}\label{4ZJholopen}
Suppose $(M, J)$ is an almost complex $4$-manifold and $\alpha$ is a non-trivial closed $J$-anti-invariant $2$-form. Then the zero set $Z$ of $\alpha$ supports a $J$-holomorphic $1$-subvariety $\Theta_{\alpha}$ in the canonical class $K_J$. 
\end{theorem}

We remark that for Hind-Tomassini's almost complex structure on $\mathbb C^2$ with infinite dimension of closed $J$-anti-invariant $2$-forms \cite{HT}, the zero sets could be explicitly determined and they are either empty or $J$-holomorphic curves diffeomorphic to $\C$. 

We can also formulate similar higher dimensional analogue for Theorems \ref{ICdim4open}-\ref{4ZJholopen}, up to the non-compact version of Question \ref{pcaholoco2}.

By applying Theorem \ref{ICdim4open} instead of Theorem \ref{ICdim4}, we obtain the non-compact version of Theorem \ref{sing4to4}.

\begin{theorem}\label{sing4to4open}
Let $u: (X, J)\rightarrow (M, J_M)$ be a somewhere immersed pseudoholomorphic map between almost complex $4$-manifolds. Then 
\begin{itemize}
\item the singularity subset of $u$ supports a $J$-holomorphic $1$-subvariety; 
\item other than countably many points $x\in M$, where $u^{-1}(x)$ is the union of a $J$-holomorphic $1$-subvariety and a $0$-subvariety, the preimage of each point is a set of countably many isolated points.
\end{itemize}
\end{theorem}

It is also natural to ask for a version of Theorem \ref{introphblowdown} and \ref{resorb} in the non-compact setting.

\begin{prop}\label{openphblowdown}
Let $u: (X, J) \rightarrow (M, J_M)$ be a somewhere immersed proper surjective pseudoholomorphic map from a connected almost complex $4$-manifold $X$ to a connected almost complex $4$-orbifold $M$ such that $J$ is almost K\"ahler and $u^{-1}(p)$ contains only one point for any $p$ in an open subset $D\subset M$. Then there exists a subset $M_1\subset M$, consisting of countably many points,  with the following significance:
\begin{enumerate}
\item The restriction $u|_{X\setminus u^{-1}(M_1)}$ is a diffeomorphism. 
\item At each smooth point of $M_1$, the preimage is an exceptional curve of the first kind.  
\item  At each orbifold point of $M_1$, the preimage is equivalent to one of the configuration of type (N). In particular, each irreducible component is a smooth rational curve and there are no cycles enclosed by irreducible components.
\end{enumerate}

\end{prop}
\begin{proof}
The singularity subset $\mathcal S_u$ is a $J$-holomorphic $1$-subvariety by Theorem \ref{sing4to4open} and its image $u(\mathcal S_u)$ is a union of $1$- and $0$-subvarieties. Hence $D\setminus \mathcal S_u\ne \emptyset$ and $u|_{D\setminus\mathcal S_u}$ is a diffeomorphism on it. Hence by connectedness of $M$, any point on $M\setminus \mathcal S_u$ only has a unique preimage and $u|_{u^{-1}(M\setminus \mathcal S_u)}$ is a diffeomorphism. The image $u(\mathcal S_u)\subset M$ contains two parts. The first part, $M_1$, consists of points whose preimage contain non-trivial $J$-holomorphic $1$-subvarieties. There are countably many such points. For any point $m\in u(\mathcal S_u)\setminus M_1$, we know by the argument of Proposition 5.9 in \cite{ZCJM}, $u^{-1}(m)$ is a single point. This completes the proof of the first statement of the theorem.

Since $u$ is proper, the preimage $u^{-1}(p)$ for any $p\in M_1$ is compact and thus a pseudoholomorphic subvariety in the original sense. Take an open ball $U$ round $p$ and since its closure is compact, $u^{-1}(\bar U)$ is compact and thus has finite topology. This implies $p$ is not an accumulation point of $M_1$. Hence we can take $U$ such that $U\cap M_1=\{p\}$ and our Grauert theorem argument for Theorem \ref{introphblowdown}(2) applies to show the second and third statements. 
\end{proof}

When $u$ is not proper, we have an almost complex $4$-manifold $X'\supset X$ and $u$ extends to a pseudoholomorphic map $u'$ from $X'$ such that $X'\setminus X$ has Hausdorff dimension at most two and $u'(X'\setminus X)\subset M_1$. It should be true that the preimage of each point of $M_1$ is an exceptional curve of the first kind in $X'$, or, in particular, each connected component of the exceptional locus is a subset of an exceptional curve of the first kind. However, our Grauert theorem argument does not directly apply since a point in $M_1$ might be accumulating point of $M_1$ and the rational surface obtained from an open ball around it could have infinite topology. A corresponding orbifold version could also be formulated. In this case, we should expect each connected component of the exceptional locus is a subset of topological blowup of curve configurations of type (N).

\section{Birational invariants for almost complex manifolds}\label{aKod}
Theorem \ref{sing4to4} suggests degree one pseudoholomorphic maps are birational morphisms in the almost complex category. We define two almost complex manifolds $M$ and $N$ to be birational to each other if there are almost complex manifolds $M_1, \cdots, M_{n+1}$ and $X_1, \cdots, X_{n}$ such that $M_1=M$ and $M_{n+1}=N$, and there are degree one pseudoholomorphic maps $f_i: X_i\rightarrow M_i$ and $g_i: X_{i}\rightarrow M_{i+1}$, $i=1, \cdots, n$.

 An important step to develop birational geometry for almost complex manifolds is to introduce and study birational invariants. The Kodaira dimension, plurigenera as well as dimensions of holomorphic $p$-forms for compact complex manifolds are very important classical birational invariants.  In this section, we will find and study their counterparts for (compact) almost complex manifolds. Most of the results are derived in \cite{CZ, CZ2} with Haojie Chen.
 
 Let $(X,J)$ be a $2n$-dimensional almost complex manifold. We have $$T^*X\otimes \mathbb C=(T^*X)^{1, 0}\oplus (T^*X)^{0,1}$$ where $(T^*X)^{1,0}$ annihilates the subspace in $TX\otimes \mathbb C$ where $J$ acts as $-i$. Write $\Lambda^{p, q}X=\Lambda^p((T^*X)^{1, 0})\otimes \Lambda^q((T^*X)^{0, 1})$. We have the canonical bundle $\mathcal K=\Lambda^{n,0}$, which is no longer holomorphic for non-integrable $J$. 
 
 To generalize holomorphic bundles and its holomorphic section, we need to introduce the notion of {\it pseudoholomorphic structure}. 
 A pseudoholomorphic structure on a complex vector bundle $E$ is given by a differential operator $\bar{\partial}_E: \Gamma(X, E)\rightarrow  \Gamma(X, (T^*X)^{0,1}\otimes E)$ which satisfies the Leibniz rule $$\bar{\partial}_E (fs)=\bar{\partial}f\otimes s+f\bar{\partial}_Es$$ where $f$ is a smooth function and $s$ is a section of $E$. If the pseudoholomorphic structure $\bar{\partial}_E$ satisfying $\bar{\partial}_E^2=0$ on a complex manifold, it is equivalent to a holomorphic structure on the complex bundle $E$ by Koszul-Malgrange theorem. In particular, any pseudoholomorphic structure on a complex vector bundle over a Riemann surface $S$ is holomorphic.
 
 Apply $\bar{\partial}=\pi^{p,1}\circ d$ to $\Lambda^{p, 0}$ and in particular $\mathcal K=\Lambda^{n,0}$, we have standard pseudoholomorphic structures $$\bar{\partial}: \Lambda^{p,0}\rightarrow \Lambda^{p,1}\cong (T^*X)^{0,1} \otimes \Lambda^{p,0},$$
 $$\bar{\partial}: \mathcal K\rightarrow \Lambda^{n,1}\cong (T^*X)^{0,1} \otimes \mathcal K.$$
 It can be extended  to a pseuholomorphic structure $\bar{\partial}_m: \mathcal K^{\otimes m}\rightarrow T^{0,1}\otimes \mathcal K^{\otimes m}$ on pluricanonical bundle for $ m\geq 2$ inductively by the product rule. 
 
 We are able to generalize plurigenera and the two (equivalent) definitions of complex Kodaira dimension to almost complex manifolds.

 Denote $\Omega^p(E)=\Lambda^{p,0}\otimes E$. The pseudoholomorphic structures on $\Lambda^{p,0}$ and $E$ gives a pseudoholomorphic structure on $\Omega^p(E)$. Define $$H^0(X, \Omega^p(E))=\{s\in \Gamma(X,\Omega^p(E))=\Omega^{p,0}(X, E): \bar{\partial}_{E} s=0\}.$$ The finiteness of the dimension of $\bar\partial$-harmonic forms are essentially pointed out in \cite{Hir}. Building on Hodge theorem for almost Hermitian manifolds, we show the following in \cite{CZ}. 

\begin{prop}\label{finsecE}
Let $E$ be a complex vector bundle with a pseudoholomrphic structure over a compact almost complex manifold $X$, then $H^0(X, \Omega^p(E))$ is finite dimensional for $0\leq p\leq n$. In particular, $H^0(X,\mathcal K^{\otimes m})$ is finite dimensional.
\end{prop}

We have a second, more geometric, description of pseudoholomorphic sections. There is a type of special almost complex structures on the total space of the complex vector bundle $E$, called {\it bundle almost complex structures}, introduced in \cite{dBT}. They show that there is a bijection between bundle almost complex structures and the pseudoholomorphic structures on $E$. We further observe, in \cite{CZ}, that a section in the kernel of a pseudoholomorphic structure $\bar{\partial}_E$ is exactly a $(J, \mathcal J)$-holomorphic section with respect to the bundle almost complex structure $\mathcal J$ corresponding to $\bar{\partial}_E$ and the almost complex structure $J$ on $M$.

With these two equivalent descriptions understood, we are able to give our definition of $(E, \mathcal J)$-genus and the first definition of Iitaka dimension, as well as their special cases - the plurigenera and  the Kodaira dimension.

\begin{definition}\label{Iv1}
 The $(E, \mathcal J)$-genus of $(X, J)$ is defined as $P_{E, \mathcal J}:=\dim H^0(X, (E, \mathcal J))$. The $m^{th}$ plurigenus of $(X,J)$ is defined to be $P_m(X, J)=\dim H^0(X, \mathcal K^{\otimes m})$.

Let $L$ be a complex line bundle $L$ with bundle almost complex structure $\mathcal J$ over $(X, J)$. The Iitaka dimension $\kappa^J(X, (L, \mathcal J))$ is defined as
$$\kappa^J(X, (L, \mathcal J))=\begin{cases}\begin{array}{cl} -\infty, &\ \text{if} \ P_{L^{\otimes m}}=0\  \text{for any} \ m\geq 0\\
\limsup_{m\rightarrow \infty} \dfrac{\log P_{L^{\otimes m}}}{\log m}, &\  \text{otherwise.}
\end{array}\end{cases}$$

The Kodaira dimension $\kappa^J(X)$ is defined by choosing $L=\mathcal K$ and $\mathcal J$ to be the bundle almost complex structure induced by $\bar{\partial}$.
\end{definition}

For the second definition of Kodaria (and Iitaka) dimensions, we use generalization of pluricanonical maps for an almost complex manifold $(X, J)$. More generally, for any complex line bundle $L$ with a bundle almost complex structure $\mathcal J$, $\Phi_{L, \mathcal J}: X\setminus B\rightarrow \mathbb CP^{N}$ is defined as $\Phi_{L, \mathcal J}(x)=[s_0(x): \cdots :s_{N}(x)]$, where $s_i$ constitute a basis of the linear space $H^0(X, (L, \mathcal J))$ and $B$ is the base locus, {\it i.e.} the set of points $x\in X$ such that $s_i(x)=0, \forall i$. It is a pseudoholomorphic map. 
When $L=\mathcal K_X^{\otimes m}$ and $\mathcal J=\mathcal J_J$ induced by $\bar{\partial}_m$, $\Phi_{L, \mathcal J}$ is denoted by $\Phi_m$ and called the pluricanonical map.

We then have another version of Iitaka dimension for almost complex manifolds. 
\begin{definition}
The Iitaka dimension $\kappa_J(X, (L, \mathcal J))$ of a complex line bundle $L$ with bundle almost complex structure $\mathcal J$ over $(X, J)$ is defined as
$$\kappa_J(X, (L, \mathcal J))=\begin{cases}\begin{array}{cl} -\infty, &\ \text{if} \ P_{L^{\otimes m}}=0\  \text{for any} \ m\geq 0\\
\max\dim \Phi_{L^{\otimes m}} (X\setminus B_{m, L, \mathcal J}), &\  \text{otherwise.}\end{array}
\end{cases}$$

The Kodaira dimension $\kappa_J(X)$ is defined by choosing $L=\mathcal K$ and $\mathcal J$ to be the bundle almost complex structure induced by $J$.
\end{definition}

The two versions $\kappa^J$ and $\kappa_J$ are equal to each other and equal to the original Iitaka dimension when $J$ is integrable and $L$ is a holomorphic line bundle.

\begin{question}\label{k=k}
Is $\kappa_J(X, (L, \mathcal J))=\kappa^J(X, (L, \mathcal J))$ for a complex line bundle $L$ with bundle almost complex structure $\mathcal J$ over $(X, J)$? In particular, is $\kappa_J(X)=\kappa^J(X)$?
\end{question}
In general, we even do not know whether $\kappa^J$ is always an integer. In dimension $4$, Question \ref{k=k}  is equivalent to: if $\kappa_J(X, (L, \mathcal J))=1$, is it true that $\kappa^J(X, (L, \mathcal J))=1$?

Lastly, there is another generalization of (real) Dolbeault cohomology groups for almost complex manifolds which is introduced earlier in \cite{LZ}. The cohomology groups  $$H_J^{\pm}(M)=\{ \mathfrak{a} \in H^2(M;\mathbb R) | \exists \; \alpha\in \mathcal Z_J^{\pm} \mbox{ such that } [\alpha] = \mathfrak{a} \},$$ where $\mathcal Z_J^{\pm}$ are the spaces of closed $2$-forms in $\Omega_J^{\pm}$.  It is proved in \cite{DLZ} that $H_J^+(M)\oplus H_J^-(M)=H^2(M; \R)$ when $\dim_{\R} M=4$. The dimensions  of  the vector spaces $H_J^{\pm}(M)$ are denoted as $h_J^{\pm}(M)$.

 To show the plurigenera, as well as the $J$-anti-invariant dimension $h_J^-(X)=\dim H_J^-(X)$ and the Hodge numbers $h^{p,0}(X)=\dim H^0(X, \Omega^p(\mathcal O))$, are birational invariants  when $\dim X=4$, a crucial step, which certainly has its independent interest, is the following Hartogs extension theorem in almost complex setting \cite{CZ}. So far it is only established in dimension $4$ by the foliation-by-disks technique. The corresponding version for closed $J$-anti-invariant forms is established in \cite{BonZ}.

\begin{theorem}\label{hartogs}Let $(X,J)$ be an almost complex $4$-manifold and $p\in X$.
\begin{enumerate}
\item \cite{CZ} Let $(E, \mathcal J)$ be a complex vector bundle with a bundle almost complex structure over $(X, J)$. Then any section in $H^0(X \! \setminus \!  \{p\}, (E, \mathcal J)|_{X\setminus \{p\}})$ extends to a section in $H^0(X, (E, \mathcal J))$.

\item \cite{BonZ}
 Suppose that $p$ is in an open subset $U\subset X$ and $\alpha$ is a closed $J$-anti-invariant $2$-form defined on $U \! \setminus \! \{ p \}$. Then $\alpha$ extends smoothly to $U$.

\end{enumerate}
\end{theorem}

A direct application is the following

\begin{theorem}\label{Kodbir}
Let $u: (X, J)\rightarrow (Y, J_Y)$ be a degree one pseudoholomorphic map between closed almost complex $4$-manifolds. Then 

\begin{itemize}

\item $P_m(X, J)=P_m(Y, J_Y)$ and thus $\kappa^{J}(X)= \kappa^{J_Y}(Y)$;

\item $\kappa_{J}(X)= \kappa_{J_Y}(Y)$;

\item $h^{p,0}(X)=h^{p,0}(Y)$ for any $0\le p\le 2$;

\item $h_J^-(X)=h_{J_Y}^-(Y)$.
\end{itemize}

In other words, all of these are birarional invariants.
\end{theorem}
Meanwhile, none of the above, except trivial cases $P_0$ and $h^{0,0}$, is a deformation, thus smooth, invariant, although all of these are when $(X, J)$ is a complex surface \cite{FQ}. On the other hand, by Theorem \ref{introphblowdown} (3), the fundamental group is an almost complex birational and topological invariant in dimension $4$. It would be interesting to show this is true in higher dimensions, even if assuming Question \ref{pcaholoco2}.  By the same argument as for Theorem \ref{sing4to4}, we know the singularity subset $\mathcal S_u$ of a degree one pseudoholomorphic map $u$ between $2n$ dimensional almost complex manifolds is a pseudoholomorphic $(n-1)$-subvariety up to Question \ref{pcaholoco2}. The map $u$ is of degree one would imply the preimage is either a single point or a higher dimensional subvariety. Thus the key part is to understand the structure of each connected component of $\mathcal S_u$.

These should be compared with symplectic birational invariants. In symplectic world, there is no definition of plurigenera and the Kodaira dimension is only defined for manifolds with dimension at most four in a numerical way. However, it is easy to see the fundamental group and the Artin-Mumford invariant Tor$H^3(X, \mathbb Z)$ are invariant under symplectic birational equivalence as they are preserved under blow-up, blow-down and symplectic deformation.

We would like to say a bit more on the almost complex Kodaira dimensions $\kappa_J$ and $\kappa^J$. Apparently, $J$ is integrable if $\kappa_J=2$. Since the images of pluricanonical maps are complex, the Kodaira dimension gives information on integrability of $J$. However, there exist non-integrable almost complex $2n$-manifolds with $\kappa_J=\kappa^J$ being any number of $\{-\infty, 0, 1, \cdots, n-1\}$ for any $n\ge 2$, as constructed in \cite{CZ} alongside many other interesting examples. In fact, one should aim to classify those almost complex manifolds with positive Kodaira dimension. For instance, we have shown that if $(X, J)$ is an almost complex $4$-manifold admitting a base-point-free pluricanonical map and $\kappa_J(X)=1$, then $X$ admits a pseudoholomorphic elliptic fibration with finitely many singular fibers. 

As a complex manifold could also be endowed with non-integrable almost complex structures and a symplectic manifold has compatible or tamed almost complex structures, we can also compare the value of $\kappa^J$ (or $\kappa_J$) with the original Kodaira dimension $\kappa^h$, or the symplectic Kodaira dimension $\kappa^s$ for $4$-dimensional symplectic manifolds \cite{Ltj}. 

\begin{theorem}[\cite{CZ2}]\label{compare}
Let $(X,J)$ be a tamed almost complex $4$-manifold. Then 
\begin{itemize}
\item $\kappa^J(X)\le \kappa^s(X)$;  
\item $\kappa^J(X)\le \kappa^h(X)$ when $X$ also admits a complex structure.
\end{itemize}
\end{theorem}
In dimension $4$, both $\kappa^s$ and $\kappa^h$ are smooth invariants. Thus we can write $\kappa^s(X)$ and $\kappa^h(X)$ instead of $\kappa^s(X, \omega)$ and $\kappa^h(X, J)$. The proof of Theorem \ref{compare} uses an interesting probabilistic combinatorial type argument to estimate the growth of the dimension of moduli spaces of curves in pluricanonical classes.

The almost complex Kodaira dimensions and plurigenera are very computable invariants.  Besides the elliptic fibration examples with large Kodaira dimensions mentioned above, \cite{CZ} also computes them for other families including almost complex structures on $4$-torus and Kodaira-Thurston manifolds as well as the standard almost complex structures on $S^6$. These computations reveal many interesting phenomenon of these numbers. Recently, explicit computations are also made on Nakamura manifolds by Cattaneo-Nannicini-Tomassini \cite{CNT}. In their paper, relations between Kodaira dimension and the curvature of the canonical connection of an almost K\"ahler manifolds are also explored.

\section{Other directions}
In this section, we will discuss several other directions of almost complex geometry. Results in Sections \ref{lap} and \ref{cone} were known in the complex setting. However, many results in Section \ref{sphere} seem new even for projective surfaces.

\subsection{Eigenvalues of Laplacian}\label{lap}
The theory of pseudoholomorphic subvarieties could also help to obtain upper bounds for the eigenvalues of the Laplacian on almost K\"ahler manifolds.

In his Warwick thesis \cite{Bon}, Louis Bonthrone obtains the following

\begin{theorem}\label{Bonth}
Let $(M^{2n}, J)$ be closed, almost K\"ahler and $\phi\in W^{1,2}_{loc}(M, \mathbb CP^m)$ a non-trivial, locally approximable pseudoholomorphic map. Then, for any almost K\"ahler metric $g$ on $M$, the eigenvalues of the Laplace-Beltrami operator satisfy $$\lambda_k(M, g)\le C(n, m)\frac{\int_M\phi^*\omega_{FS}\wedge \omega_g^{n-1}}{\int_M\omega_g^n}k, \, \, \, \forall k\ge 1.$$
\end{theorem}

A map $\phi\in W^{1,2}_{loc}(M,  N)$ is called locally approximable if for any ball $\bar B\subset M$ there exists a sequence $\phi_i\in C^{\infty}(M, N)$ such that $\phi_i\rightarrow \phi$ strongly in $W^{1,2}(B, N)$. It is proved in \cite{BCDH} that $\phi\in W^{1,2}_{loc}(M,  N)$ is locally approximable if and only if $d(\phi^*\alpha)=0, \forall \alpha\in \mathcal Z^2(N)$, holds in the sense of currents. 

This type of upper bound could date back to Bourguignon-Li-Yau's bound for the first eigenvalue for a given K\"ahler metric on a projective manifold \cite{BLY} where the above constant $C(n, m)$ can be chosen as $4n\frac{m+1}{m}$ if there exists a holomorphic immersion $\phi: M^{2n}\rightarrow \mathbb CP^m$. Recently Kokarev \cite{Kok} has extended their result to the same type of upper bounds on the $k$-th eigenvalue for K\"ahler manifolds which admits a non-trivial holomorphic map  $\phi: M\rightarrow \mathbb CP^m$. 

However, we would like to remark that even in the K\"ahler case, Theorem \ref{Bonth} says somethings new. In fact, it applies to all rational maps $\phi: M\dashrightarrow \mathbb CP^m$ since they are locally approximable and belong to  $W^{1,2}_{loc}(M, \mathbb CP^m)$.

Theorem \ref{Bonth} also applies to complex vector bundles over a compact almost complex manifold. For a pseudoholomorphic complex vector bundle $E$ which is globally generated by pseudoholomorphic sections, we have the inequality

$$\lambda_k(M, g)\le C(n, m)\frac{(c_1(E)\cup  [\omega_g]^{n-1}, [M])}{([\omega_g]^n, M)}k, \, \, \, \forall k\ge 1.$$
This is derived by applying Theorem \ref{Bonth} to the map $K_E: M\rightarrow \mathbb CP^m$ which is obtained as the composition of  the Kodaira map $\kappa_E: M\rightarrow Gr(r, \dim H^0(X, E))$ and the Pl\"uker embedding of the latter. In particular, it is applicable to our examples with $\kappa^J\ge 1$ in \cite{CZ}. In fact, the above statement is also true when the locus, where $E$ is not globally generated, is of complex codimension at least two, {\it e.g.} a complex line bundle whose base locus is of complex codimension at least two.

One can also obtain a version of Theorem \ref{Bonth} for pseudoholomorphic subvarieties of almost K\"ahler manifolds, whose K\"ahler version has been proved in \cite{Kok}. 

\begin{theorem}\cite{Bon}
Let $\Sigma^{2n}\subset M^{2n+2l}$ be a irreducible pseudoholomorphic subvariety in almost K\"ahler manifold $M$ with pseudoholomorphic map $\phi: M\rightarrow \mathbb CP^m$ restricting non-trivially on $\Sigma$. Then, for any almost K\"ahler metric $g$ on $M$, the Laplace eigenvalues of $(\Sigma, g_{\Sigma})$ satisfy 
$$\lambda_k(\Sigma, g_{\Sigma})\le C(n, m)\frac{\int_{\Sigma}\phi^*\omega_{FS}\wedge \omega_g^{n-1}}{\int_M\omega_g^n}k, \,\, \, \forall k\ge 1.$$
\end{theorem}

\subsection{Cones of (co)homology classes for almost complex manifolds}\label{cone}
Our results in previous sections could be viewed as developing algebraic geometry on almost complex manifolds. Before \cite{ZCJM}, we have already studied several aspects of algebraic geometry for almost complex manifolds in several papers including \cite{LZrc, p=h, Zmod}. However, the argument to prove most of these results, which would be partially reviewed in the following two subsections, only works in dimension four as we use Seiberg-Witten theory in an essential way.

The study of various cones, {\it e.g.} the cone of curves, the ample cone, and the K\"ahler cone, has played a very important role in algebraic geometry. Remarkable results include Mori's cone theorem, Nakai-Moishezon's and Kleiman's ampleness criteria, and Demailly-Paun's description of K\"ahler cone. In \cite{p=h}, we study the cone of curves and the corresponding duality results for almost complex manifolds. 

For  tamed almost complex manifold $(M, J)$, we define the cone of curves
$A_J(M)=\{\sum a_i[C_i]| a_i> 0\},$ where $C_i$ are irreducible $J$-holomorphic curves on $M$. When $\dim M=4$, by taking the Poincar\'e duality, we know $A_J(M)$ is a cone in the vector space $ H_J^+(M)\subset H^2(M; \mathbb R)$. The tameness of $J$ implies $A_J(M)$ does not contain a straight line through the origin. If we denote $A_J^{K_J\ge 0}(M)=\{C\in A_J(M)| K_J\cdot C\ge 0\}$ to be the ``positive" part of the curve cone, we have the Mori's cone theorem for an arbitrary tamed almost complex $4$-manifold \cite{p=h}.

\begin{theorem}\label{coneintro}
Let $(M, J)$ be a tamed almost complex $4$-manifold. Then $$\overline{A}_J(M)=\overline{A}_J^{K_J\ge 0}(M)+\sum \mathbb R^+[L_i]$$ where $L_i\subset M$ are countably many smooth irreducible rational curves such that $-3\le K_J\cdot [L_i]<0$ which span the extremal rays $\mathbb R^+[L_i]$ of $\overline{A}_J(M)$.

Moreover, for any $J$-almost K\"ahler symplectic form $\omega$ and any given $\epsilon>0$, there are only finitely many extremal rays with $(K_J+\epsilon[\omega])\cdot [L_i]\le 0$.

In addition, an irreducible curve $C$ is an extremal rational curve if and only if 

\begin{enumerate}
\item $C$ is a $-1$ rational curve;
\item $M$ is a minimal ruled surface or $\mathbb CP^2\# \overline{\mathbb CP^2}$, and $C$ is a fiber;
\item $M=\mathbb CP^2$ and $C$ is a projective line.
\end{enumerate}
\end{theorem}

This cone theorem is a powerful tool in studying pseudoholomorphic curves on $4$-manifolds. For example, the first step to establish it is a simple but very useful lemma (Lemma 2.1 in \cite{p=h}), which says any irreducible $J$-holomorphic curve $C$ with $K_J\cdot [C]<0$ has to be a $-1$ rational curve. This has been used in different settings, {\it e.g.} pseudoholomorphic foliation, (equivariant) almost complex birational geometry, and length of Wahl singularities. A quick corollary of the cone theorem also implies for any symplectic $4$-manifold which contains smooth $-1$ spheres and is not diffeomorphic to $\mathbb CP^2\# \overline{\mathbb CP^2}$, there exists at least one smooth $J$-holomorphic $-1$ rational curve for any tamed $J$.

We have also established some partial results on the almost complex version of the Nakai-Moishezon and the Kleiman type dualities between the curve cone and the almost K\"ahler cone $$\mathcal K_J^c=\{[\omega]\in H^2(M; \mathbb R)| \omega \mbox{ is compatible with }J\}.$$ When $b^+(M)=1$, it is shown in \cite{LZ} that $\mathcal K_J^c$ is equal to the tame cone $\mathcal K_J^t=\{[\omega]\in H^2(M; \mathbb R)| \omega \mbox{  tames }J\}$. 

To state our results, we define the positive cone $\mathcal P=\{e\in H^2(M;\mathbb R)|e\cdot e >0\}$,  the positive dual of $A_J(M)$ (resp. $\overline{A}_J(M)$), $A_J^{\vee, >0}(M)$ (resp. $\overline{A}_J^{\vee, >0}(M)$), where the duality is taken within $H_J^+(M)$, and $\mathcal P_J=A_J^{\vee, >0}(M) \cap \mathcal P$. Clearly,  $\mathcal K_J^c\subset \mathcal P_J$. We would like to know whether $\mathcal K_J^c=\mathcal K_J^t= \mathcal P_J=\overline{A}_J^{\vee, >0}(M)$ when $b^+=1$ and whether $\mathcal K_J^c$ is a connected component of $\mathcal P_J$ when $b^+>1$ (Question 1.4 in \cite{p=h}). 

\begin{theorem}\label{rr}
If $J$ is almost K\"ahler on $M=S^2\times S^2$ or $\mathbb CP^2\# k\overline{\mathbb CP^2}$ with $k\le 9$, then $$\mathcal K_J^c=\mathcal K_J^t= \mathcal P_J=\overline{A}_J^{\vee, >0}(M).$$
\end{theorem}

It is also true that $\mathcal K_J^t=\mathcal K_J^c= \mathcal P_J$ for a generic tamed $J$ when $b^+=1$. This result could be used to guarantee the correctness of many applications of tamed $J$-inflation \cite{Mc2, Bus} under the recent weakening of this technique (see {\it e.g.} \cite{ALLP}).

To the author, the statements of Mori cone theorem and Nakai-Moishezon or Kleiman duality should not only work in dimension $4$, although certainly new techniques are needed in higher dimensions.

\subsection{Spherical classes on almost complex $4$-manifolds}\label{sphere}
The study of minimal $2$-spheres in compact Riemannian manifolds is a very important topic. Sacks-Uhlenbeck's existence of minimal immersions of $2$-spheres was a remarkable milestone \cite{SU}. In algebraic geometry, the study of holomorphic spheres is ubiquitous. This always leads to the most geometric and successful part of the stories. For symplectic manifolds, particularly in dimension $4$, spheres also play an important role. For instance, when a closed symplectic $4$-manifold $M$ contains an embedded symplectic spheres of nonnegative self-intersection, then it is symplectomorphic to a blowup of either $\mathbb CP^2$ or a sphere bundle over surface \cite{McD90}.  The same conclusion also follows from the existence of a smoothly embedded sphere with nonnegative self-intersection and infinite order in $H_2(M, \mathbb Z)$ \cite{Lpams}.

In the almost complex category, as other results in this paper, we would like to state results that work for an arbitrary (tamed) almost complex structure, rather than requiring genericity. We would also need the class to be Gromov-Witten stable in most situations, {\it i.e.} there are always pseudoholomorphic subvarieties in this given homology class. 

For an almost complex structure $J$, we define the $J$-genus (or virtual genus) of $e\in H^2(M, \mathbb Z)$ as $g_J(e):=\frac{1}{2}(e\cdot e+K_J\cdot e)+1$. A $K_J$-spherical class is a class $e$ which could be represented by a smoothly embedded sphere and $g_J(e)=0$. A spherical class is Gromov-Witten stable only when $e^2=-2-K_J\cdot e\ge -1$. We call a $K_J$-spherical class $E$ with $E^2=K_J\cdot E=-1$ an exceptional curve class.

\begin{theorem}\label{1E}
Let $M$ be a symplectic $4$-manifold which is not diffeomorphic to $\mathbb CP^2\#k\overline{\mathbb CP^2}, k\ge 1$. Then for any tamed $J$, there is a unique $J$-holomorphic subvariety in any exceptional curve class $E$, whose irreducible components are smooth rational curves of negative self-intersection. Moreover, this $J$-holomorphic subvariety is an exceptional curve of the first kind.
\end{theorem}

The existence of $J$-holomorphic subvarieties follows directly from Seiberg-Witten-Taubes theory. Hence, the key statement  of this result is the uniqueness. When $M=\mathbb CP^2\#k\overline{\mathbb CP^2}, k\ge 1$, a $J$-holomorphic subvariety in an exceptional curve class need not to be an exceptional curve of the first kind. Moreover, there are exceptional curve classes such that the moduli space of $J$-holomorphic subvarieties in such a class $E$ is of positive dimension and some representatives have higher genus components when $k\ge 8$ as noticed in \cite{p=h, Zmod}. Theorem \ref{1E} was first established for irrational ruled surfaces, {\it i.e.} smooth $4$-manifolds diffeomorphic to blowups of sphere bundles over Riemann surfaces with positive genus, in \cite{Zmod}. The non-ruled situation is  resolved in the appendix of \cite{CZ2} using different techniques. 

If the spherical class has nonnegative self-intersection, the manifold $M$ must be a rational or ruled surface. Moreover,  the only spherical class with nonnegative self-intersection on irrational ruled surfaces is the positive fiber class. However, {\it a priori}, we do not know how the subvarieties in the fiber class behave. For example, whether do we have the exotic behaviour appeared in the rational surfaces situation, and do we always have smooth subvarieties in this class for any tamed almost complex structure? These could be understood very well in the ruled surfaces case and we derive the following structural results for an arbitrary tamed almost complex structure. 

\begin{theorem}\label{intro2}
Let $M$ be an irrational ruled surface of base genus $h\ge 1$. Then for any tamed $J$ on $M$, 
\begin{enumerate}
\item there is a unique subvariety in the positive fiber class $T$ passing through any given point;
\item the moduli space $\mathcal M_T$ is homeomorphic to $\Sigma_h$, and there are finitely many reducible varieties;
\item every irreducible rational curve is an irreducible component of a subvariety in class $T$.
\end{enumerate}
\end{theorem}

The above results could be understood as describing linear systems for spherical classes in the almost complex setting. It is remarkable that even if the statement of Theorem \ref{intro2} is purely $4$-dimensional, it has been used in a substantial way to study symplectic Fano $6$-manifolds \cite{LP} where no genericity of almost complex structures can be assumed on the symplectic reduction.  

Besides Seiberg-Witten-Taubes theory and in particular the wall-crossing formula for manifolds with $b^+=1$ \cite{LLwall}, we use the $J$-nef technique developed in \cite{LZrc} to prove Theorems \ref{1E} and \ref{intro2}. We call a class $e\in H^2(M, \mathbb Z)$ $J$-nef if it pairs non-negatively with any $J$-holomorphic subvariety. The exotic phenomenon of the existence of higher genus irreducible components (and higher dimensional moduli) of a spherical class, as in the rational counterexamples of Theorem \ref{1E}, cannot happen for $J$-nef (spherical) classes. 

\begin{theorem}\label{genus bound}
Suppose   $e$ is a $J$-nef class with $g_J(e)\geq 0$. Then  $$g_J(e)\ge t(\Theta):=\sum_i g_J(e_{C_i})$$ for any connected subvariety $\Theta=\{(C_i,m_i)\}$ in the class $e$.  
\end{theorem}

When the  $J$-genus is zero, we have the following more precise description.

\begin{theorem} \label{emb-comp}  Suppose  $e$ is a  $J$-nef class with $g_J(e)=0$. 
Let   $\Theta=\{(C_i,m_i)\}$ be a $J$-holomorphic subvariety in the class $e$. 
\begin{itemize}
\item 
If $\Theta$ is connected, then   each  irreducible component  of $\Theta$  is a smooth rational curve, and $\Theta$ is a tree configuration.
\item
If $J$ is tamed, then $\Theta$ is connected. 

\item Let $l_e:=\max\{e\cdot e+1, 0\}$. If $\Theta$ is connected and contains at least two irreducible components, then 
\begin{equation*}\label{red-dim'}
\sum_{i=1}^n m_i l_{e_{C_i}}\leq l_e-1.
\end{equation*}
\end{itemize}
\end{theorem}
\begin{cor} \label{embsphere} Suppose $J$ is a tamed almost complex structure  
and   $e$ is a class   represented by a smooth $J$-holomorphic rational curve. 
Then for  any $J$-holomorphic subvariety $\Theta$ in the class $e$, 
each irreducible component of $\Theta$ is a smooth  rational curve. 
\end{cor}
The subvarieties with the equality $\sum_{i=1}^n l_{e_{C_i}}\leq l_e-1$ holds are classified in \cite{LZrc}. These results are proved by introducing combinatorial moves on the weighted graph underlying the subvarieties and showing the relevant quantities like total genus and sum of dimensions are changed monotonically when applying these moves. This is a very effective method to treat problems involving algebraic aspect of the adjunction formula. 

The strategy to prove Theorem \ref{intro2} is to first show the positive fiber class is nef for any tamed $J$ by Seiberg-Witten theory. Then Theorem \ref{emb-comp} implies there is at least one embedded $J$-holomorphic curve in class $T$. A geometric argument reduces the moduli space description to uniqueness of curves passing through a given point. This uniqueness statement and that of Theorem \ref{1E} for irrational surfaces follows from a generalization of Gromov's positivity of intersection to general pseudoholomorphic subvarieties, which says that, under nefness assumption, two distinct subvarieties in the same class cannot intersect at too many points on each irreducible component, see Lemma 2.5 in \cite{Zmod} and Lemma 4.18 in \cite{LZ-generic}.

This set of techniques also leads to linear systems on spherical classes with positive self-intersection. 

\begin{theorem}\label{cpl}
Let $J$ be a tamed almost complex structure on a rational surface $M$. Suppose $e$ is a primitive class and represented by a smooth $J$-holomorphic sphere. Then $\mathcal M_e$ is homeomorphic to $\mathbb CP^l$ where $l=\max\{0, e\cdot e+1\}$.
\end{theorem}

\end{document}